\documentclass{article}
\usepackage{fancyhdr,graphicx}
\usepackage{amsfonts}
\usepackage{amssymb}
\usepackage{amsthm}
\usepackage{newlfont}
\usepackage{amsmath}
\usepackage{multicol}
\usepackage{subfigure}
\usepackage{xcolor}
\usepackage{blkarray}
\usepackage{amscd}
\usepackage[utf8]{inputenc}

\usepackage{amsmath}
\usepackage{blkarray}
\usepackage{color}
\newcommand{\JW}[1]{f_{1}}
\newcommand{\coeff}[2]{\coefficient_{\in \JW{1}^{2}}}

\DeclareMathOperator{\coefficient}{coeff}

\newcommand{\db}[1]{\left(\left(1\right)\right)}

\newtheorem{theorem}{Theorem}[section]
\newtheorem{lemma}[theorem]{Lemma}
\newtheorem{remark}[theorem]{Remark}
\newtheorem{corollary}[theorem]{Corollary}
\newtheorem{proposition}[theorem]{Proposition}
\newtheorem{definition}[theorem]{Definition}

\newtheorem{example}[theorem]{Example}
\newtheorem{problem}[theorem]{Problem}
\graphicspath{{figures/}}

\usepackage{psfrag}

\usepackage{mathrsfs}
\usepackage{subfig}
\usepackage{picinpar}

\begin{document}
\title
	{Matrix Quasi-tree Theorem}
	
	\author{Qingying Deng\\
		{\small School of Mathematics and Computational Science, Xiangtan University, P. R. China}\\
		Xian'an Jin\\
		{\small School of Mathematical Sciences, Xiamen University, P. R. China}\\
		Qi Yan\footnote{Corresponding author.}\\
		{\small School of Mathematics and Statistics, Lanzhou University, P. R. China}\\
        Yexiang Yan\\
		{\small School of Mathematics and Computational Science, Xiangtan University, P. R. China}\\
		{\small{Email: qingying@xtu.edu.cn, xajin@xmu.edu.cn,yanq@lzu.edu.cn, linmu394@163.com}}
	}
	\date{}
	
	\maketitle
\begin{abstract}
Building on prior work that established Matrix Quasi-tree Theorems for special embedded graphs, in this paper, we develop a comprehensive theory applicable to all embedded graphs. We introduce symbolic skew-adjacency matrices and reduction maps as key innovations, and prove that a specific polynomial derived from these matrices encodes all spanning quasi-trees of a bouquet. This result provides a complete analogue of the Matrix Tree Theorem for topological graph theory, with applications to quasi-tree enumeration in both orientable and non-orientable embedded graphs.
\end{abstract}

\section{Introduction}
The classical Matrix Tree Theorem, established by Kirchhoff \cite{1847Kirchhoff} and further developed by Tutte \cite{MTT}, provides an elegant combinatorial interpretation of determinants for enumerating spanning trees in graphs.

Let $G$ be a finite loopless multigraph with vertex set $V(G)$ and edge set $E(G)$. For an edge $e \in E(G)$, we denote its set of endpoints by $v(e) \subseteq V(G)$. We associate a variable $x_e$ to each edge $e$, and for any subset $F \subseteq E(G)$, we define the monomial
\[
x_F = \prod_{e \in F} x_e.
\]
The \emph{Kirchhoff matrix} (or Laplacian matrix) $\mathbf{K} = (\ell_{ij})$ is defined as a symmetric matrix indexed by the vertices of $G$, with entries given by
$$\ell_{ij}=-\sum_{\substack{e\in E(G) \\ v(e)=\{i, j\}}}x_e, \text{\ \ if\
} i\ne j,   \text{\ \ \ and\ \ \ }
 \ell_{ii}=\sum_{\substack{e\in E(G) \\ i\in v(e)}}x_e.$$
Since each row of $\mathbf{K}$ sums to zero, its determinant vanishes. However, the determinant of the submatrix $\mathbf{K}^{(p)}$ obtained by deleting the $p$th row and column is independent of the choice of $p$. This determinant defines the \emph{Kirchhoff polynomial} $\mathcal{K}_G$ in the variables $x_e$.

The Matrix Tree Theorem \cite[Theorem VI.29]{MTT} states that the Kirchhoff polynomial expands as
\begin{equation*}
\mathcal{K}_G = \det \mathbf{K}^{(p)} = \sum_T x_{E(T)},
\end{equation*}
where the sum is taken over all spanning trees of $G$. Each monomial $x_{E(T)}$ appears with coefficient 1 and corresponds to a spanning tree of the graph. For further details, we refer to \cite{Masbaum}.

Recent developments in graph polynomials \cite{2018Butler, 2011Champanerkar, 2011Vignes} and delta-matroid \cite{ChunJCTA, Chun} reveal a fundamental correspondence: spanning quasi-trees in topological graph theory are the natural analogue of spanning trees in abstract graph theory.

Merino et al.~\cite{Merino2023} established an analogue of Kirchhoff's Matrix Tree Theorem for orientable embedded graphs, replacing the count of spanning trees with that of spanning quasi-trees. This result can alternatively be derived from the work of Macris and Pule \cite[Theorem 1]{Macris} on Eulerian circuit enumeration (see also \cite{1997Lauri, 1999Lauri}).

\begin{theorem}[\cite{Bouchet87,1997Lauri,Macris, Merino2023}]\label{theorem1}
Let $B$ be an orientable bouquet with edge set $[n]= \{1, 2, \dots, n\}$ and signed rotation $\sigma$, and let $\mathbf{A_{\sigma}^u}$ be its unsymbolic skew-adjacency matrix. Then the number of spanning quasi-trees of $B$ is
\[
\tau(B) = \det(I_n + \mathbf{A_{\sigma}^u}).
\]
\end{theorem}

This result was subsequently extended by Deng et al. \cite{Deng24} to bouquets containing exactly one non-orientable loop using delta-matroid theory. This progression naturally leads to the following problem:
\begin{problem}
Establish a Matrix Quasi-tree Theorem for all embedded graphs.
\end{problem}

In this paper, we introduce the novel concepts of symbolic skew-adjacency matrices and reduction maps $f$ (defined in Section 3), and use them to establish a comprehensive Matrix Quasi-tree Theorem for all embedded graphs.

Throughout this paper, for any polynomial $P \in \mathbb{Z}[x_A : A \subseteq [n]]$, we define $P \bmod 2$ to be the polynomial obtained by reducing each coefficient modulo 2.

\begin{theorem}[Matrix Quasi-tree Theorem]\label{thm:main}

Let $B$ be a bouquet with edge set $[n]$ and signed rotation $\sigma$, and let $\mathbf{A_{\sigma}^s}$ be its symbolic skew-adjacency matrix. Then
\begin{eqnarray*}
f(\det(I_n + \mathbf{A_{\sigma}^s})) \bmod 2= \sum_X x_X,
\end{eqnarray*}
where the sum is taken over all the edge sets of spanning quasi-trees in $B$.
Consequently,
\[
\tau(B) = f(\det(I_n + \mathbf{A_{\sigma}^s})) \bmod 2 \big|_{x_X = 1}.
\]
\end{theorem}

An example demonstrating Theorem \ref{thm:main} is given as follows.

\begin{figure}[!htbp]
\begin{center}
\includegraphics[width=9cm]{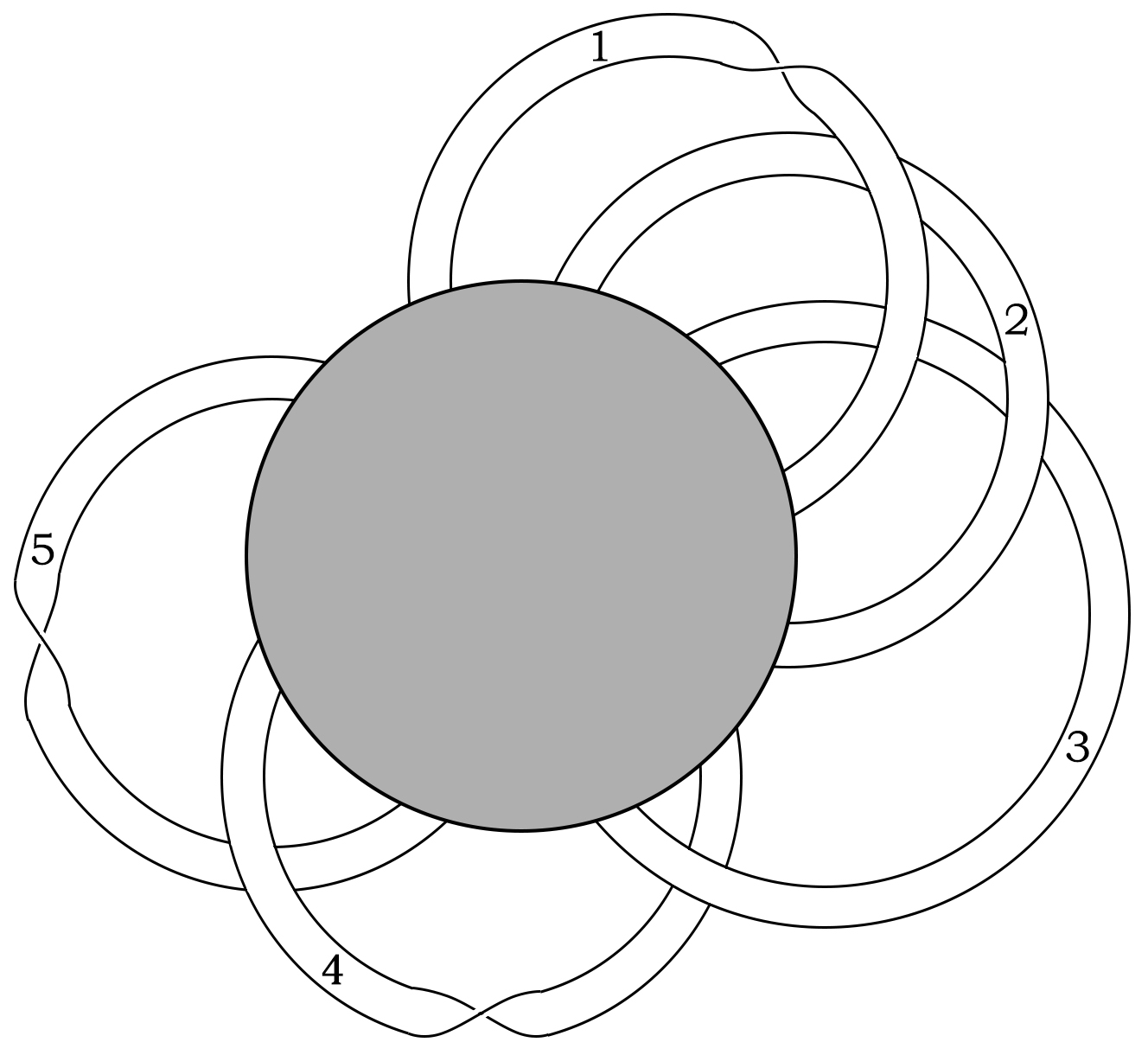}
\caption{A signed rotation of the bouquet is $[-1^a, 2^a, 3^a, 1^b, 2^b, -4^a,3^b,-5^a,4^b,5^b]$.}
\label{Fig02}
\end{center}
\end{figure}
\begin{example}
Let $B$ be a bouquet with signed rotation $\sigma$ as depicted in Figure \ref{Fig02}:
\[\sigma=
[-1^a, 2^a, 3^a, 1^b, 2^b, -4^a,3^b,-5^a,4^b,5^b].
\]
 Its symbolic skew-adjacency matrix is
\[
\mathbf{A_{\sigma}^s} = \begin{bmatrix}
x_{11} & x_{12} & x_{13} & 0 &0\\
-x_{12} & 0 & x_{23}  & 0 &0\\
-x_{13} & -x_{23} & 0 & x_{34} &0\\
0 & 0 & -x_{34} & x_{44} & x_{45}\\
0 & 0 &0  & -x_{45} &x_{55}
\end{bmatrix}.
\]
Then
\begin{align*}
\det(I_5 + \mathbf{A_{\sigma}^s})&=x_{11} x_{23}^{2} x_{44} x_{55} + x_{11} x_{23}^{2} x_{44} + x_{11} x_{23}^{2} x_{45}^{2} + x_{11}x_{23}^{2}x_{55}+ x_{11} x_{23}^{2}+x_{11}x_{34}^{2} x_{55}\\
&+ x_{11} x_{34}^{2} + x_{11} x_{44} x_{55} + x_{11} x_{44}+ x_{11} x_{45}^{2}+ x_{11} x_{55}+ x_{11} + x_{12}^{2} x_{34}^{2} x_{55}\\
& + x_{12}^{2} x_{34}^{2} + x_{12}^{2} x_{44} x_{55}+x_{12}^{2} x_{44}+ x_{12}^{2} x_{45}^{2}+x_{12}^{2} x_{55} + x_{12}^{2}+ x_{13}^{2} x_{44} x_{55}\\
&+ x_{13}^{2} x_{44} + x_{13}^{2} x_{45}^{2} + x_{13}^{2} x_{55}+ x_{13}^{2}+ x_{23}^{2} x_{44} x_{55} + x_{23}^{2} x_{44} + x_{23}^{2} x_{45}^{2}  \\
&+ x_{23}^{2} x_{55} + x_{23}^{2} + x_{34}^{2} x_{55} + x_{34}^{2} + x_{44} x_{55} + x_{44} + x_{45}^{2} + x_{55} + 1.
\end{align*}
Therefore,
\begin{align*}
f(\det(I_5 + \mathbf{A_{\sigma}^s})) &=x_{12345}+x_{1234}+x_{12345}+x_{1235} + x_{123}+x_{1345}\\
&\quad+ x_{134} + x_{145} + x_{14}+ x_{145}+ x_{15}+ x_{1} + x_{12345}\\
&\quad + x_{1234} + x_{1245}+x_{124}+ x_{1245}+x_{125} + x_{12}+x_{1345}\\
&\quad+ x_{134} + x_{1345} + x_{135}+ x_{13}+ x_{2345} + x_{234} + x_{2345}  \\
&\quad+ x_{235} + x_{23} + x_{345} + x_{34} + x_{45} + x_{4} + x_{45} + x_{5} + 1\\
&=1 + x_1 + x_{12} + x_{123} + 2x_{1234} + 3x_{12345} + x_{1235} \\
&\quad + x_{124} + 2x_{1245} + x_{125} + x_{13} + 2x_{134} + 3x_{1345} + x_{135} \\
&\quad + x_{14} + 2x_{145} + x_{15} + x_{23} + x_{234} + 2x_{2345} + x_{235} \\
&\quad + x_{34} + x_{345} + x_4 + 2x_{45} + x_5.
\end{align*}
Reducing modulo $2$, we obtain
\begin{align*}
f(\det(I_5 + \mathbf{A_{\sigma}^s})) \bmod 2 &= 1 + x_1 + x_{12} + x_{123} + x_{12345} + x_{1235} + x_{124} + x_{125} \\
&\quad + x_{13} + x_{1345} + x_{135} + x_{14} + x_{15} + x_{23} + x_{234} \\
&\quad + x_{235} + x_{34} + x_{345} + x_4 + x_5.
\end{align*}
Setting all $x_{X} = 1$ for $X\subseteq [n]$ recovers $\tau(B)=20$ and the set of spanning quasi-trees of $B$ is
\begin{align*}
\{&\emptyset, \{1\}, \{12\}, \{123\}, \{12345\}, \{1235\}, \{124\}, \{125\}, \{13\}, \{1345\}, \\
&\{135\}, \{14\}, \{15\}, \{23\}, \{234\}, \{235\}, \{34\}, \{345\}, \{4\}, \{5\}\}.
\end{align*}
\end{example}

This example clearly demonstrates that the reduction modulo 2 in Matrix Quasi-tree Theorem is essential. As seen in the computation, the polynomial $f(\det(I_5 + \mathbf{A_{\sigma}^s}))$ contains coefficients greater than 1 (specifically, coefficients 2 and 3 appear in various terms). The modulo 2 operation ensures that each spanning quasi-tree contributes exactly once to the count, eliminating the overcounting that would otherwise occur due to these higher coefficients. This phenomenon underscores the necessity of the $\bmod~2$ reduction in the main theorem statement.

A bouquet is a ribbon graph with a single vertex. Crucially, any connected ribbon graph $G$ can be transformed into a bouquet by applying partial duality with respect to the edge set of any spanning quasi-tree \cite{EM}. Since partial duality preserves the number of spanning quasi-trees \cite{Chun}, we obtain $\tau(G) = \tau(B)$ for the resulting bouquet $B$. This transformation allows us to extend our Matrix Quasi-tree Theorem to general connected ribbon graphs, as detailed in Corollary~\ref{Cor1}.

\section{Preliminaries}
There are different ways to describe an embedded graph. In this paper, we will use the concept of a ribbon graph.

\begin{definition}[\cite{bollobas}]
A {\it ribbon graph} $G=(V(G), E(G))$ is a $($orientable or non-orientable$)$ surface with boundary,
represented as the union of two sets of topological discs, a set $V(G)$ of vertices, and a set $E(G)$ of edges,
subject to the following restrictions.
\begin{description}
\item[(1)] The vertices and edges intersect in disjoint line segments;
\item[(2)] Each such line segment lies on the boundary of precisely one vertex and precisely one edge;
\item[(3)] Every edge contains exactly two such line segments.
\end{description}
\end{definition}

A ribbon graph \( H = (V(H), E(H)) \) is a \emph{ribbon subgraph} of \( G = (V(G), E(G)) \) if \( H \) can be obtained from \( G \) by deleting vertices and edges. If \( V(H) = V(G) \), then \( H \) is a \emph{spanning ribbon subgraph} of \( G \).
For a subset \( A \subseteq E(G) \), the ribbon subgraph \emph{induced by \( A \)}, denoted \( G|_A \), is obtained from \( G \) by deleting all edges not in \( A \), and then deleting any isolated vertices.
A ribbon graph is \emph{non-orientable} if it contains a subgraph homeomorphic to a M\"{o}bius band; otherwise, it is \emph{orientable}.
An edge \( e \) is a \emph{loop} if it is incident to only one vertex. A loop is \emph{non-orientable} if its corresponding ribbon, together with the vertex, forms a M\"{o}bius band; otherwise, it is \emph{orientable}.

Partial duality, introduced by Chmutov~\cite{Chmutov}, provides a powerful tool for studying signed Bollob\'{a}s-Riordan polynomials and knot polynomials. For a ribbon graph $G$ and $A\subseteq E(G)$,  the \emph{partial dual}, $G^{\delta(A)}$,  of $G$ with respect to $A$ is the ribbon graph obtained from $G$ by gluing a disc to $G$ along each boundary component of the spanning ribbon subgraph $(V (G), A)$ (such discs will be the vertex-discs of $G^{\delta(A)}$), removing the interiors of all vertex-discs of $G$ and keeping the edge-ribbons unchanged.

For an edge $e \in E(G)$, the \emph{contraction} of $e$ is defined as:
$$G/ e:=G^{\delta(e)}\backslash e.$$ Table~\ref{Fig3} from~\cite{EM} illustrates the local transformations of partial duality, deletion, and contraction on an edge of a ribbon graph.
\begin{table}
  \centering
\caption{Operations on an edge $e$ (highlighted in bold) of a ribbon graph}\label{Fig3}
  \includegraphics[width=15cm]{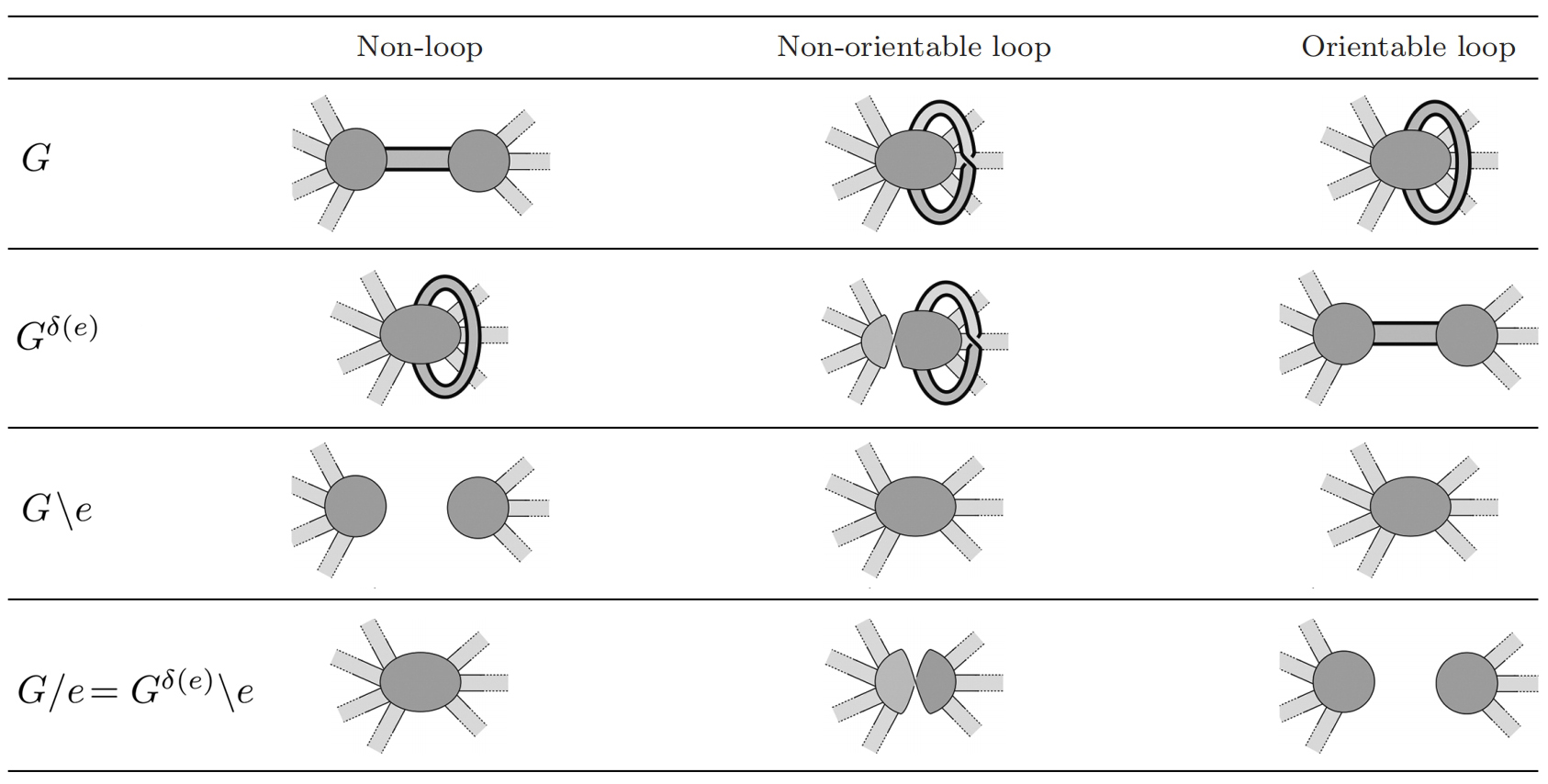}\\
\end{table}

Quasi-trees serve as the topological counterpart to trees in ribbon graph theory, and the substitution of ``tree" by ``quasi-tree" in the cycle matroid definition naturally gives rise to a delta-matroid.

A set system is a pair $D=(E, \mathcal{F})$, where $E$ is a finite set called the \emph{ground set}, and $\mathcal{F}$ is a collection of subsets of $E$. The elements of $\mathcal{F}$ are called \emph{feasible sets}. A set system is \emph{proper} if $\mathcal{F} \neq \emptyset$.

Following Bouchet~\cite{AB1}, a \emph{delta-matroid} is a proper set system $D=(E, \mathcal{F})$ satisfying the following symmetric exchange axiom: for any $X, Y \in \mathcal{F}$ and any $u \in X \Delta Y$, there exists $v \in X \Delta Y$ (possibly $v=u$) such that $X \Delta \{u, v\} \in \mathcal{F}$. Here, $X \Delta Y = (X \cup Y) \setminus (X \cap Y)$ denotes the symmetric difference of sets.

For $A \subseteq E$, the \emph{twist} of $D$ with respect to $A$, denoted $D \ast A$, is defined as:
\[
D \ast A = (E, \{A \Delta X \mid X \in \mathcal{F}\}).
\]

A ribbon graph is called a \emph{quasi-tree} if it has exactly one boundary component.
The number of spanning quasi-trees in a connected ribbon graph $G$ is denoted by $\tau(G)$. Remarkably, the collection of edge sets of all spanning quasi-trees in a ribbon graph forms the feasible sets of a delta-matroid on $E(G)$.

\begin{definition}[\cite{ChunJCTA}]
    Let $G = (V(G), E(G))$ be a connected ribbon graph, and define
    \[
    \mathcal{F}(G) := \{F \subseteq E(G) \mid F \text{ is the edge set of a spanning quasi-tree of } G\}.
    \]
    The delta-matroid of $G$ is defined as $D(G) = (E(G), \mathcal{F}(G))$.
\end{definition}

Chun et al.~\cite{Chun} established the equivalence between twisting and partial duality at the delta-matroid level.

\begin{proposition}[\cite{Chun}]\label{prop:tau-and-twist}
    Let $G$ be a ribbon graph and $A \subseteq E(G)$. Then
    \[
    D(G) \ast A = D(G^{\delta(A)}).
    \]
\end{proposition}

\section{Symbolic skew-adjacency matrices}

\begin{definition}
\normalfont
Let $B$ be a bouquet with edge set $[n]$. A \emph{signed rotation} of $B$ is a cyclic sequence of length $2n$ encoding the order of half-edges around the vertex. Each loop $i \in [n]$ appears twice in the sequence, with one occurrence labeled $i^a$ and the other occurrence labeled $i^b$.
The signs assigned to these labels are determined by the orientability of the corresponding loop: for an orientable loop, $i^a$ and $i^b$ are assigned identical signs, while for a non-orientable loop, they are assigned opposite signs.
\end{definition}

\begin{remark}
In the signed rotation representation, the positive sign is conventionally omitted in the notation. A bouquet may admit multiple signed rotations due to different choices in the starting point of the cyclic sequence, the direction of rotation (clockwise or counterclockwise), and the assignment of $a$ and $b$ labels to the two occurrences of each loop. These variations all yield equivalent representations of the same bouquet. See Figure~\ref{Fig02} for an example.
\end{remark}
%\begin{figure}[!htbp]
%\begin{center}
%\includegraphics[width=9cm]{f01.jpg}
%\caption{A signed rotation of the bouquet is $(1^a, 2^a, -1^b, 3^a, 2^b, -3^b, 4^a, 4^b)$.}
%\label{f01}
%\end{center}
%\end{figure}

\begin{definition}
\normalfont
Let $B$ be a bouquet with edge set $[n]$ and signed rotation $\sigma$. The \emph{symbolic skew-adjacency matrix} $\mathbf{A_{\sigma}^s}$ of $B$ with respect to $\sigma$ is an $n \times n$ matrix over the polynomial ring $\mathbb{Z}[x_{ij}]$, where $x_{ij}$ are commuting indeterminates for all $i,j \in [n]$. The entries of $\mathbf{A_{\sigma}^s}$ are defined as follows:

\begin{itemize}
\item \textbf{Diagonal entries ($i = j$):}
\[
\mathbf{A_{\sigma}^s}_{ii} =
\begin{cases}
x_{ii}, & \text{if the signs of $i^a$ and $i^b$ are opposite}, \\
0, & \text{if the signs of $i^a$ and $i^b$ are the same}.
\end{cases}
\]

\item \textbf{Off-diagonal entries ($i < j$):}
Consider the relative cyclic order of $i^a$, $i^b$, $j^a$, $j^b$ in $\sigma$ (ignoring their signs):
\[
\mathbf{A_{\sigma}^s}_{ij} =
\begin{cases}
x_{ij}, & \text{if the cyclic order is $i^a, j^a, i^b, j^b$}, \\
-x_{ij}, & \text{if the cyclic order is $i^a, j^b, i^b, j^a$}, \\
0, & \text{otherwise}.
\end{cases}
\]

\item \textbf{Skew-symmetry:}
For $i > j$, define $\mathbf{A_{\sigma}^s}_{ji} = -\mathbf{A_{\sigma}^s}_{ij}$.
\end{itemize}

We define two related matrices:
\begin{itemize}
\item The \emph{unsymbolic skew-adjacency matrix} $\mathbf{A_{\sigma}^u}$ is obtained from $\mathbf{A_{\sigma}^s}$ by substituting $x_{ij} = 1$ for all $i,j \in [n]$.
\item The \emph{adjacency matrix} $\mathbf{M}$ is defined by $\mathbf{M}_{ij} = |\mathbf{A_{\sigma}^u}_{ij}|$ for all $i,j \in [n]$.
\end{itemize}
\end{definition}

\begin{remark}\label{remark1}
\begin{description}
\item[(1)]
It is important to note that the symbolic skew-adjacency matrix $\mathbf{A_{\sigma}^s}$ and the unsymbolic skew-adjacency matrix $\mathbf{A_{\sigma}^u}$ depend not only on the bouquet $B$ but also on the specific signed rotation $\sigma$. Different signed rotations of the same bouquet may yield different skew-adjacency matrices.

Specifically, the dependence on $\sigma$ manifests as follows:
\begin{itemize}
\item Changing the starting point of the cyclic sequence leaves $\mathbf{A_{\sigma}^s}$ and $\mathbf{A_{\sigma}^u}$ unchanged.
\item Reversing the cyclic direction flips the signs of off-diagonal entries (i.e., $x_{ij}$ becomes $-x_{ij}$ and $-x_{ij}$ becomes $x_{ij}$ in $\mathbf{A_{\sigma}^s}$, while $1$ becomes $-1$ and $-1$ becomes $1$ in $\mathbf{A_{\sigma}^u}$).
\item Swapping the $a$ and $b$ labels for a loop multiplies the corresponding row and column by $-1$ in both $\mathbf{A_{\sigma}^s}$ and $\mathbf{A_{\sigma}^u}$.
\end{itemize}

However, the adjacency matrix $\mathbf{M}$ is independent of the choice of signed rotation. Since this matrix is defined by taking absolute values of the entries of $\mathbf{A_{\sigma}^u}$, it depends only on the bouquet $B$ itself. Notably, $\mathbf{M}$ is precisely the adjacency matrix of the signed intersection graph of $B$ \cite{Moffatt2019, 2022Yan}.

\item[(2)]
Let $B$ be a bouquet with edge set $[n]$ and signed rotation $\sigma$, and let $\mathbf{A_{\sigma}^s}$ be its symbolic skew-adjacency matrix. By the Leibniz formula, the determinant of $\mathbf{A_{\sigma}^s}$ is given by
\[
\det(\mathbf{A_{\sigma}^s}) = \sum_{\pi \in S_n} \operatorname{sgn}(\pi) \prod_{i=1}^n \mathbf{A_{\sigma}^s}_{i,\pi(i)},
\]
where $S_n$ denotes the symmetric group on $[n]$, and each monomial corresponds to a permutation $\pi$. Since each entry $\mathbf{A_{\sigma}^s}_{ij}$ is an element of the ring $\mathbb{Z}[x_{ij}]$, it follows that $\det(\mathbf{A_{\sigma}^s}) \in \mathbb{Z}[x_{ij}]$.
Furthermore, the determinant of $I_n + \mathbf{A_{\sigma}^s}$ can be expanded as a sum over all principal submatrices:
\[
\det(I_n + \mathbf{A_{\sigma}^s}) = \sum_{X \subseteq [n]} \det(\mathbf{A_{\sigma}^s}[X]) \in \mathbb{Z}[x_{ij}],
\]
where $\mathbf{A_{\sigma}^s}[X]$ denotes the principal submatrix of $\mathbf{A_{\sigma}^s}$ induced by $X \subseteq [n]$, with the convention that $\det(\mathbf{A_{\sigma}^s}[\emptyset]) = 1$.
\end{description}
\end{remark}

\begin{definition}
\normalfont
We define a reduction map $f: \mathbb{Z}[x_{ij}] \to \mathbb{Z}[x_A : A \subseteq [n]]$ as follows.
The target ring has generators $\{x_A : A \subseteq [n]\}$ with multiplication $x_A \cdot x_B = x_{A \cup B}$
and unit $x_\emptyset$. For a monomial $m = \prod_{r=1}^k x_{i_r j_r}^{e_r}$ with $e_r\in\mathbb{Z}_{\geq 1}$, define
\[
f(m) = x_A, \quad \text{where } A = \bigcup_{r=1}^k \{i_r,j_r\}.
\]
This extends linearly to polynomials: for $P = \sum_m \alpha_m m \in \mathbb{Z}[x_{ij}]$ with $\alpha_m \in \mathbb{Z}$,
\[
f(P) = \sum_m \alpha_m f(m).
\]
The map $f$ preserves coefficients while replacing each monomial by a single variable indexed by the union of all indices appearing in its factors.
\end{definition}

\section{Proof of the Matrix Quasi-tree Theorem}

Before proving the Matrix Quasi-tree Theorem (Theorem~\ref{thm:main}), we establish some key lemmas.
The next lemma combines results from two references. For orientable bouquets,
Merino et al.~\cite{Merino2023} established the following result using directed connected maps and principally unimodular matrices.
Deng et al.~\cite{Deng24} later extended this result to bouquets with exactly
one non-orientable loop.

\begin{lemma}[\cite{Deng24}, \cite{Merino2023}]\label{orient}
Let $B$ be a bouquet with edge set $[n]$ and signed rotation $\sigma$, and let $\mathbf{A_{\sigma}^u}$ be its unsymbolic skew-adjacency matrix. For any subset $X \subseteq [n]$, if $B$ has at most one non-orientable loop, then
\[
\det(\mathbf{A_{\sigma}^u}[X]) =
\begin{cases}
1, & \text{if } X \text{ is the edge set of a spanning quasi-tree of } B, \\
0, & \text{otherwise}.
\end{cases}
\]
\end{lemma}

\begin{remark}
A direct consequence of Lemma~\ref{orient} is that for bouquets with at most one non-orientable loop, the polynomial $f(\det(I_n + \mathbf{A_{\sigma}^s}))$ has all coefficients equal to 1. Therefore, the modulo 2 reduction in Theorem~\ref{thm:main} becomes unnecessary in these cases.
\end{remark}

The pivot operation is defined for matrices over arbitrary fields (see, e.g., \cite{2000TMJ}).

\begin{definition}
Let $\mathbf{A}$ be a square matrix over a field \( \mathbb{K}\), with rows and columns labelled (in the same order) by a set $E$. For any $X \subseteq E$, let $\overline{X} = E \setminus X$. Assume without loss of generality (after reordering if needed) that $X$ labels the first $|X|$ rows and columns.
Then $\mathbf{A}$ has the block form
\[
\mathbf{A} =
\begin{blockarray}{ccc}
 & X & \overline{X} \\
\begin{block}{c[cc]}
X & P & Q \\
\overline{X} & R & S \\
\end{block}
\end{blockarray}
\]
where $P = \mathbf{A}[X]$. If $P$ is non-singular, the \textit{pivot} of $\mathbf{A}$ on $X$, denoted $\mathbf{A}*X$, is defined as the matrix
\[
\mathbf{A}*X =
\begin{blockarray}{ccc}
 & X & \overline{X} \\
\begin{block}{c[cc]}
X & P^{-1} & -P^{-1}Q \\
\overline{X} & RP^{-1} & S- RP^{-1}Q \\
\end{block}
\end{blockarray}.
\]
\end{definition}

\begin{lemma}[\cite{Moffatt2019}]\label{lemma:matrixpivot}
Let $B$ be a bouquet with edge set $[n]$, and let $\mathbf{M}(B)$ be the adjacency matrix of $B$ over $\mathbb{F}_2$. If $e_1\in [n]$ is a non-orientable loop, then the pivot matrix \( \mathbf{M}(B) * \{e_1\} \) coincides with the adjacency matrix  of \( B^{\delta(e_1)} \), i.e., \[\mathbf{M}(B) * \{e_1\}=\mathbf{M}(B^{\delta(e_1)}),\]
where $\mathbf{M}(B^{\delta(e_1)})$ is the adjacency matrix of $B^{\delta(e_1)}$.
\end{lemma}

\begin{theorem}\label{nonorient-genernal}
Let $B$ be a bouquet with edge set $[n]$ and signed rotation $\sigma$, and let $\mathbf{M}$ be the adjacency matrix of $B$. Then for any subset $X \subseteq [n]$,
\begin{eqnarray}\label{eq10}
\det(\mathbf{M}[X]) \bmod 2 =
\begin{cases}
1, & \text{if } X \text{ is the edge set of a spanning quasi-tree of } B, \\
0, & \text{otherwise}.
\end{cases}
\end{eqnarray}
\end{theorem}
\begin{proof}
We first consider the case when $B$ is orientable. Let $\mathbf{A_{\sigma}^u}$ be the unsymbolic skew-adjacency matrix of $B$ with respect to $\sigma$. Since $\mathbf{M}_{ij} = |\mathbf{A_{\sigma}^u}_{ij}|$ and $\mathbf{A_{\sigma}^u}_{ij} \in \{-1, 0, 1\}$, we have $\mathbf{M}_{ij} \equiv \mathbf{A_{\sigma}^u}_{ij} \pmod{2}$ for all $i,j$. Hence, $\mathbf{A_{\sigma}^u} \equiv \mathbf{M} \pmod{2}$. This implies that for any subset $X \subseteq [n]$, the principal submatrices satisfy $\mathbf{A_{\sigma}^u}[X] \equiv \mathbf{M}[X] \pmod{2}$, and therefore
\[
\det(\mathbf{A_{\sigma}^u}[X]) \equiv \det(\mathbf{M}[X]) \pmod{2}.
\]
By Lemma \ref{orient}, $\det(\mathbf{A_{\sigma}^u}[X])$ is 1 if and only if $X$ is the edge set of a spanning quasi-tree of $B$, and 0 otherwise. Thus, the same holds for $\det(\mathbf{M}[X]) \bmod 2$.

For the case when $B$ has at least one non-orientable loop, we proceed by induction on $k = |X|$, the size of the edge subset.

\textbf{Base Cases:}

\textbf{Case $k = 0$:} When $X = \emptyset$, the matrix $\mathbf{M}[X]$ is empty and its determinant is defined to be $1$ by convention. The empty edge set corresponds to the spanning quasi-tree of the bouquet consisting of a single isolated vertex (which has one boundary component). Thus, $\det(\mathbf{M}[X]) \bmod 2 = 1$, and equation \eqref{eq10} holds.

\textbf{Case $k = 1$:} When $X$ contains a single edge, $\mathbf{M}[X]$ is a $1 \times 1$ matrix. If the edge is a non-orientable loop, then $\mathbf{M}[X] = [1]$ and $\det(\mathbf{M}[X]) = 1$. If the edge is an orientable loop, then $\mathbf{M}[X] = [0]$ and $\det(\mathbf{M}[X]) = 0$. A single edge of a bouquet forms a spanning quasi-tree if and only if it is non-orientable. Therefore, $\det(\mathbf{M}[X]) \bmod 2 = 1$ if and only if $X$ is the edge set of a spanning quasi-tree, satisfying equation \eqref{eq10}.

\textbf{Inductive Step:} Assume the result holds for all edge subsets of size less than $k$, where $k \geq 2$. Now consider a set $X$ with $|X| = k$.

Let $B|_X$ denote the bouquet obtained by restricting $B$ to the subset $X$ of edges. We consider two subcases based on the orientability of $B|_X$.

\textbf{Case 1: $B|_X$ is orientable.}
By the argument at the beginning of the proof, equation \eqref{eq10} holds.

\textbf{Case 2: $B|_X$ is non-orientable.}
Let \( e_1 \in X \) be a non-orientable loop. Note that the adjacency matrix of the restricted bouquet \( B|_X \) is exactly the principal submatrix \( \mathbf{M}[X] \).

We partition the matrix $\mathbf{M}[X]$, placing $e_1$ first:
$$
\mathbf{M}[X] = \begin{bmatrix}
P & Q \\
Q^\top & R
\end{bmatrix},
$$
where $P = \mathbf{M}[X][\{e_1\}] = [1]$ (since $e_1$ is non-orientable). Hence, $P$ is invertible. Using the block matrix determinant formula:
\begin{equation*}
\det\begin{bmatrix}
P & Q \\
Q^\top & R
\end{bmatrix} = \det(P) \cdot \det(R - Q^\top P^{-1}Q). \label{blockdet}
\end{equation*}
Since $P = [1]$, we have $P^{-1} = [1]$, so
\[
\det(\mathbf{M}[X]) = \det(R - Q^\top Q).
\]
Thus,
\begin{equation}
\det(\mathbf{M}[X]) \bmod 2 = \det(R - Q^\top Q) \bmod 2. \label{detmod2}
\end{equation}

Recall that by Lemma~\ref{lemma:matrixpivot}, the pivot operation on $\mathbf{M}[X]$ with respect to $\{e_1\}$ yields the adjacency matrix of $(B|_X)^{\delta(e_1)}$ over $\mathbb{F}_2$.  In our partitioned matrix $\mathbf{M}[X]= \begin{bmatrix} P & Q \\ Q^\top & R \end{bmatrix}$,
the pivot operation gives:
\[
\mathbf{M}[X] * \{e_1\} = \begin{bmatrix} P^{-1} & -P^{-1}Q \\ Q^\top P^{-1} & R - Q^\top P^{-1}Q \end{bmatrix}.
\]
Since $P = [1]$ and $P^{-1} = [1]$, the bottom-right block becomes $R - Q^\top Q$.
This submatrix $R - Q^\top Q$ corresponds to the adjacency matrix of the bouquet obtained by first taking the partial dual of $B|_X$ with respect to $e_1$ to get $(B|_X)^{\delta(e_1)}$,
and then deleting $e_1$ to obtain  $(B|_X)^{\delta(e_1)}\backslash \{e_1\}$, that is, $B|_X/\{e_1\}$.
Therefore, over $\mathbb{F}_2$, $R - Q^\top Q$ is the adjacency matrix of  $B|_X/\{e_1\}$.

By the induction hypothesis, equation \eqref{eq10} holds for the bouquet $B|_X/\{e_1\}$ on the smaller edge set $X \setminus \{e_1\}$, that is,
\begin{eqnarray}\label{eq11}
\det(R - Q^\top Q) \bmod 2 =
\begin{cases}
1, & \text{if } B|_X/\{e_1\} \text{ is a spanning quasi-tree} , \\
0, & \text{otherwise}.
\end{cases}
\end{eqnarray}

From Table \ref{Fig3}, we see
that $F\in \mathcal{F}(B|_X)$ and $e_1\in F$ if and only if $F\backslash e_1\in \mathcal{F}((B|_X)^{\delta(e_1)})$ if and only if $F\backslash e_1\in \mathcal{F}((B|_X)^{\delta(e_1)}\backslash e_1)=\mathcal{F}((B|_X)/e_1)$.
Therefore, $B|_X$ is a spanning quasi-tree if and only if  $B|_X/\{e_1\}$ is a spanning quasi-tree.
Then by equations \eqref{detmod2} and \eqref{eq11}, we conclude
\begin{eqnarray*}\label{eq12}
\det(\mathbf{M}[X]) \bmod 2 =
\begin{cases}
1, & \text{if } B|_X~\text{is a spanning quasi-tree}, \\
0, & \text{otherwise}.
\end{cases}
\end{eqnarray*}
Hence,
\begin{eqnarray*}
	\det(\mathbf{M}[X]) \bmod 2=\left\{\begin{array}{ll}
			1, & \mbox{if}~X ~\mbox{is the edge set of a spanning quasi-tree of}~B,\\
			0, & \mbox{otherwise}.
		\end{array}\right.
	\end{eqnarray*}
\end{proof}

\begin{lemma}
\label{thm:bouquet-det}
Let $B$ be a bouquet with edge set $[n]$ and signed rotation $\sigma$, and let $\mathbf{A_{\sigma}^s}$ be its symbolic skew-adjacency matrix.
For any subset $X \subseteq [n]$, we have
 \begin{equation*}\label{eq:dets-mod2}
        f\big(\det(\mathbf{A_{\sigma}^s}[X])\big) \bmod 2 =
        \begin{cases}
            x_X, & \text{if $X$ is the edge set of a spanning quasi-tree of $B$}, \\
            0, & \text{otherwise}.
        \end{cases}
    \end{equation*}
\end{lemma}

\begin{proof}
By the Leibniz formula:
\[
\det(\mathbf{A_{\sigma}^s}[X]) = \sum_{\pi \in S_X} \operatorname{sgn}(\pi) \prod_{i \in X} \mathbf{A_{\sigma}^s}_{i,\pi(i)},
\]
where $S_X$ denotes the symmetric group on $X$.
Applying the reduction map $f$ to each term:
\[
f\left(\operatorname{sgn}(\pi) \prod_{i \in X} \mathbf{A_{\sigma}^s}_{i,\pi(i)}\right) = \operatorname{sgn}(\pi)  \left(\prod_{i \in X} \mathbf{A_{\sigma}^u}_{i,\pi(i)}\right)  x_X.
\]
Summing over all $\pi \in S_X$ yields:
\[
f(\det(\mathbf{A_{\sigma}^s}[X])) = \det(\mathbf{A_{\sigma}^u}[X]) \cdot x_X.
\]
Taking modulo 2:
\[
f(\det(\mathbf{A_{\sigma}^s}[X])) \bmod 2 = \left(\det(\mathbf{A_{\sigma}^u}[X]) \bmod 2\right) \cdot x_X.
\]
Since $\det(\mathbf{A_{\sigma}^u}[X]) \equiv \det(\mathbf{M}[X]) \pmod{2}$, and by Theorem~\ref{nonorient-genernal}, it follows that
\[
f(\det(\mathbf{A_{\sigma}^s}[X])) \bmod 2 =
\begin{cases}
x_X, & \text{if $X$ is the edge set of a spanning quasi-tree of $B$}, \\
0, & \text{otherwise}.
\end{cases}
\]
\end{proof}

\noindent\textbf{Proof of Matrix Quasi-tree Theorem (Theorem~\ref{thm:main}).}
By the Leibniz formula, $\det(I_n+\mathbf{A_{\sigma}^s})$ expands as a sum over all principal submatrices:
\[
\det(I_n+\mathbf{A_{\sigma}^s}) = \sum_{X \subseteq [n]} \det(\mathbf{A_{\sigma}^s}[X]),
\]
where $\mathbf{A_{\sigma}^s}[X]$ denotes the principal submatrix indexed by $X$ and $\det(\mathbf{A_{\sigma}^s}[\emptyset]) = 1$.

Since the map $f$ is linear, it commutes with summation:
\[
f\left(\det(I_n+\mathbf{A_{\sigma}^s})\right) = f\left(\sum_{X \subseteq [n]} \det(\mathbf{A_{\sigma}^s}[X])\right) = \sum_{X \subseteq [n]} f\left(\det(\mathbf{A_{\sigma}^s}[X])\right).
\]
Reducing both sides modulo 2 yields:
\[
f\left(\det(I_n+\mathbf{A_{\sigma}^s})\right) \bmod 2 = \sum_{X \subseteq [n]} \left( f\left(\det(\mathbf{A_{\sigma}^s}[X])\right) \bmod 2 \right).
\]

By Lemma \ref{thm:bouquet-det}, $f\left(\det(\mathbf{A_{\sigma}^s}[X])\right) \bmod 2$ is nonzero (and equal to $x_X$) if and only if $X$ is the edge set of a spanning quasi-tree. Consequently, in the sum above, the coefficient of $x_X$ is 1 precisely for such subsets $X$, and 0 otherwise.

The number of spanning quasi-trees, $\tau(B)$, is the number of subsets $X$ for which the coefficient of $x_X$ is 1. This is obtained by evaluating the polynomial at $x_X = 1$ for all $X \subseteq [n]$, which sums these coefficients:
\[
\tau(B) = f(\det(I_n + \mathbf{A_{\sigma}^s})) \bmod 2 \big|_{x_X = 1}.
\]
\hfill\boxed{}

\begin{remark}
The polynomial $f(\det(I_n + \mathbf{A_{\sigma}^s}))$ is invariant under the choice of signed rotation $\sigma$, depending only on the bouquet $B$. This follows from Theorem~\ref{thm:main}, since the set of spanning quasi-trees of $B$ is independent of the specific signed rotation representation.
\end{remark}

\begin{corollary}\label{Cor1}
Let $G$ be a connected ribbon graph with edge set $[n]$, $T$ be the edge set of any spanning quasi-tree of $G$, and $\sigma$ be a signed rotation of $G^{\delta(T)}$. Let $\mathbf{A_{\sigma}^s}$ be the symbolic skew-adjacency matrix of $G^{\delta(T)}$ with respect to $\sigma$. Then
\[
f(\det(I_n + \mathbf{A_{\sigma}^s})) \bmod 2 = \sum_X x_X,
\]
where the sum is taken over all subsets $X \subseteq [n]$ such that $X \Delta T$ is the edge set of a spanning quasi-tree of $G$.
Consequently, the number of spanning quasi-trees is
\[
\tau(G) = f(\det(I_n + \mathbf{A_{\sigma}^s})) \bmod 2 \big|_{x_X = 1}.
\]
\end{corollary}

\begin{proof}
Since $T$ is the edge set of a spanning quasi-tree of $G$, the partial dual $G^{\delta(T)}$ is a bouquet. By Theorem~\ref{thm:main} applied to $G^{\delta(T)}$, the coefficient of $x_X$ in $f(\det(I_n + \mathbf{A_{\sigma}^s})) \bmod 2$ is $1$ if and only if $X$ is the edge set of a spanning quasi-tree of $G^{\delta(T)}$.

By Proposition~\ref{prop:tau-and-twist}, $D(G^{\delta(T)}) = D(G) * T$, so the edge sets of spanning quasi-trees of $G^{\delta(T)}$ are exactly the sets $F \Delta T$ where $F$ is the edge set of a spanning quasi-tree of $G$. Therefore, $X$ is the edge set of a spanning quasi-tree of $G^{\delta(T)}$ if and only if $X \Delta T$ is the edge set of a spanning quasi-tree of $G$.

By Theorem~\ref{thm:main}, the polynomial evaluation at $x_X = 1$ for all $X \subseteq [n]$ gives the number of spanning quasi-trees of $G^{\delta(T)}$. This number equals $\tau(G)$ via the bijection.
\end{proof}

\section*{Acknowledgements}
This work is supported by NSFC (Nos. 12471326, 12571379), and partially supported
by the the 111 Project (No. D23017), the Excellent Youth Project of Hunan Provincial
Department of Education, P. R. China (No. 23B0117).

\end{document}